\documentclass[12pt]{colt2019} 

\usepackage{algorithm}
\usepackage{algorithmic}
\usepackage{caption}

\usepackage{graphicx}
\usepackage{subcaption}
\usepackage{stackengine}

\newcommand*{\R}{{\mathbb R}}

\newcommand*{\e}{\varepsilon}
\newcommand*{\la}{\langle}
\newcommand*{\ra}{\rangle}

\usepackage[colorinlistoftodos,textsize=scriptsize]{todonotes}

\title[Optimal Tensor Methods for Smooth Convex Optimization]{
Optimal Tensor Methods in Smooth Convex and Uniformly Convex Optimization}
\usepackage{times}

\coltauthor{%
 \Name{Alexander Gasnikov} \Email{gasnikov@yandex.ru}\\
 \addr Moscow Institute of Physics and Technology, Institute for Information Transmission Problems, National Research University Higher School of Economics
 \AND
 \Name{Pavel Dvurechensky} \Email{pavel.dvurechensky@gmail.com}\\
 \addr Weierstrass Institute for Applied Analysis and Stochastics, Institute for Information Transmission Problems 
 \AND
 \Name{Eduard Gorbunov} \Email{eduard.gorbunov@phystech.edu}\\
 \addr Moscow Institute of Physics and Technology
 \AND
 \Name{Evgeniya Vorontsova} \Email{vorontsovaea@gmail.com}\\
 \addr Far Eastern Federal University
  \AND
 \Name{Daniil Selikhanovych} \Email{selihanovich.do@phystech.edu}\\
 \addr Moscow Institute of Physics and Technology, Institute for Information Transmission Problems
 \AND
 \Name{C\'esar A. Uribe} \Email{cauribe@mit.edu}\\
 \addr Massachusetts   Institute   of Technology%
}

\begin{document}

\maketitle
 \begin{center}
     \Large September 2, 2018\footnote{The first version of this paper appeared on September 2, 2018 in Russian. In the current version we present a translation into English of the main derivations and extend the analysis from the case of strongly convex objective to the case of uniformly convex objectives and add the numerical analysis of our results.}
 \end{center}

\begin{abstract}%
  We consider convex optimization problems with the objective function having Lipshitz-continuous $p$-th order derivative, where $p\geq 1$. We propose a new tensor method, which closes the gap between the lower  $O\left(\e^{-\frac{2}{3p+1}} \right)$ and upper $O\left(\e^{-\frac{1}{p+1}} \right)$ iteration complexity bounds for this class of optimization problems.
  We also consider uniformly convex functions, and show how the proposed method can be accelerated under this additional assumption. Moreover, we introduce a $p$-th order condition number which naturally arises in the complexity analysis of tensor methods under this assumption. 
  Finally, we make a numerical study of the proposed optimal method and show that in practice it is faster than the best known accelerated tensor method. We also compare the performance of tensor methods for $p=2$ and $p=3$ and show that the 3rd-order method is superior to the 2nd-order method in practice. 
\end{abstract}


\section{Introduction}
\label{S:Intro}


In this paper, we consider the unconstrained convex optimization problem
\begin{equation}\label{eq:pr_st}
f\left( x \right)\to \mathop {\min }\limits_{x\in {\R}^n},
\end{equation}
where $f$ has $p$-th Lipschitz-continuous derivative with constant $M_p$.
For $p=1$, first-order methods are commonly used to solve this problem, i.e., gradient descent. The lower bound for the complexity of these methods was proposed in \citep{nemirovsky1983problem,nesterov2004introduction}, and an optimal method was introduced in \citep{nesterov1983method}.
The case of $p=2$, i.e., Newton-type methods, was well understood only recently. A nearly optimal method was proposed in \citep{nesterov2008accelerating}, an optimal method was proposed in \citep{monteiro2013accelerated}, and a lower bound was obtained in \citep{agarwal2018lower,arjevani2018oracle}.

The idea of using higher order derivatives (starting from $p \geq 3$) in optimization is known at least since 1970's, see \cite{hoffmann1978higher-order}. Recently this direction of research became of interest from the point of view of complexity bounds.
In the unpublished preprint \cite{baes2009estimate}, extending the estimating functions technique of \cite{nesterov2004introduction}, proposes accelerated high-order (tensor) methods for convex problems with complexity $O\left(\left(\frac{M_pR^{p+1}}{\e}\right)^{\frac{1}{p+1}} \right)$, where $p\geq 1$, $\e$ is the accuracy of the obtained solution $\hat{x}$, i.e., $f(\hat{x}) - f^* \leq \e$, $M_p$ is the Lipschitz constant of the $p$-th derivative, and $R$ is an estimate for the distance between a starting point and the closest solution.  Nevertheless, the author doubts that the obtained methods are implementable since the auxiliary problem on each iteration is possibly non-convex. \cite{agarwal2018lower,arjevani2018oracle} construct lower complexity bounds $O\left(\left(\frac{M_pR^{p+1}}{\e}\right)^{\frac{2}{5p+1}}\right)$ and  $O\left(\left(\frac{M_pR^{p+1}}{\e}\right)^{\frac{2}{3p+1}} \right)$ respectively for the case $f$ having Lipschitz $p$-th derivative and conjecture that the upper bound can be improved. \cite{nesterov2018implementable} proposes implementable tensor methods showing that an appropriately regularized Taylor expansion of a convex function is again a convex function, thus making auxiliary problems on each iteration of the tensor methods tractable. The author also provides an accelerated scheme with complexity bound $O\left(\left(\frac{M_pR^{p+1}}{\e}\right)^{\frac{1}{p+1}} \right)$, shows that the complexity of each iteration for $p=3$ is of the same order as for the case $p=2$, and conjectures the existence of an optimal scheme with complexity bound $O\left(\left(\frac{M_pR^{p+1}}{\e}\right)^{\frac{2}{3p+1}} \right)$.

The optimal method for the case $p=1$  has complexity $O\left(\left(\frac{M_1R^{2}}{\e}\right)^{\frac{1}{2}} \right)$ \citep{nesterov1983method} and for $p=2$ has the complexity $O\left(\left(\frac{M_2R^{3}}{\e}\right)^{\frac{2}{7}} \right)$ \citep{monteiro2013accelerated}, but the question of existence of optimal methods for $p\geq 3$ remains open. In this paper we extend the framework of \cite{monteiro2013accelerated} and propose optimal tensor methods for all $p\geq 1$. Our approach is also based on regularized Taylor step of \cite{nesterov2018implementable}, and, thus, our optimal method for $p=2$ is different from \cite{monteiro2013accelerated}.

We also consider problem \eqref{eq:pr_st} under additional assumption that $f$ is uniformly convex, i.e., there exist $2 \leq q \leq p+1$ and $\sigma_q >0$ s.t.
\[
    f(y) \geq f(x) + \la \nabla f(x) , y-x \ra + \frac{\sigma_q}{q}\|y-x\|_2^q, \forall x,y \in Q.
\]
Under this additional assumption, we show, how the restart technique can be applied to accelerate our method to obtain complexity 
\[
O \left( \left( \frac{M_p}{\sigma_{p+1}} \right)^{\frac{2}{3p+1}} \log_2 \frac{\Delta_0}{\e} \right), q=p+1; \quad O \left( \left( \frac{M_p (\Delta_0)^{\frac{p+1-q}{q}}}{\sigma_q^{\frac{p+1}{q}}} \right)^{\frac{2}{3p+1}} + \log_2 \frac{\Delta_0}{\e} \right), q < p+1,
\]
where $f(x_0) -f^* \leq \Delta_0$. This bound suggests a natural generalization of first- and second-order condition number \citep{nesterov2008accelerating}. If $f$ is such that $q=p+1$, then the complexity of our algorithm depends only logarithmically on the starting point and is proportional to
\[
\left( \gamma_p \right)^{\frac{2}{3p+1}},
\]
where $\gamma_p = \frac{M_p}{\sigma_{p+1}}$ is the $p$-th order condition number. \cite{nemirovsky1983problem,nesterov2004introduction} and \cite{arjevani2018oracle} propose lower bounds for particular cases of strongly convex functions (i.e., $q=2$) with $p=1$ and $p=2$ respectively. Our upper bounds match them.

As a related work, we also mention \cite{birgin2017worst-case,cartis2017improved}, who study complexity bounds for tensor methods for finding approximate stationary points with the main focus on non-convex optimization, which we do not consider in our work. Also the work in \citep{wibisono2016variational} considers tensor methods from the variational perspective and obtains similar bounds to those in \cite{baes2009estimate}. 
The first version of this paper appeared in arXiv on September 2, 2018.
In December 2018, two months after that, \cite{jiang2018loptimal,bubeck2018near-optimal} proposed an algorithm, which is very similar to our Algorithm~\ref{alg:MSN}. Unlike them, we also analyze the case of uniformly convex functions and propose an algorithm, which is faster in this case, see our Algorithm~\ref{alg:rest-MSN}. Moreover, we are the first to make a numerical study of tensor methods for $p=3$ and show that they work in practice.

\textbf{Our contributions.}
\begin{itemize}
    \item We propose a new optimal tensor method and analyze its iteration complexity.
    \item We generalize this method for the case of uniformly convex objectives and propose a definition of $p$-th order condition number.
    \item We make a numerical study of the proposed method and show that our optimal method is faster than accelerated tensor method \cite{nesterov2018implementable} in practice. We also compare the performance of tensor methods for $p=2$ and $p=3$ and show that the 3rd-order method is superior to the 2nd-order method in practice. 
\end{itemize}

\textbf{Notations and generalities.} For $p \geq 1$, we denote by 
$
\nabla^p f(x)[h_1,...,h_p] 
$
the directional derivative of function $f$ at $x$ along directions $h _i\in \R^n$, $i = 1,...,p$. $\nabla^p f(x)[h_1,...,h_p] $ is symmetric $p$-linear form and its norm is defined as 
\[
\|\nabla^p f(x)\|_2 = \max_{h_1,...,h_p \in \R^n} \{ \nabla^p f(x)[h_1,...,h_p] : \|h_i\|_2 \leq 1, i = 1,...,p\}
\]
or equivalently
\[
\|\nabla^p f(x)\|_2 = \max_{h \in \R^n} \{ |\nabla^p f(x)[h,...,h]| : \|h\|_2 \leq 1, i = 1,...,p \}.
\]
Here, for simplicity, $\|\cdot\|_2$ is standard Euclidean norm, but our algorithm and derivations can be generalized for the Euclidean norm given by general a positive semi-definite matrix $B$.
We consider convex, $p$ times differentiable on $\R$ functions  satisfying Lipschitz condition for $p$-th derivative
\begin{equation}
    \|\nabla^p f(x) - \nabla^p f(y)\|_2 \leq M_p \|x-y\|_2, x,y \in \R^n.
\end{equation}

\section{Optimal Tensor Method}

Given a function $f$, numbers $p \geq 1$ and $M \geq 0$, define
\begin{equation}
\label{eq4}
T_{p,M}^f \left( x \right)\in \mbox{Arg}\mathop {\min }\limits_{y\in {\rm 
R}^n} \left\{ {\sum\limits_{r=0}^p {\frac{1}{r!}\nabla ^rf\left( x \right)} 
\underbrace {\left[ {y-x,...,y-x} \right]}_r+\frac{M}{\left( {p+1} 
\right)!}\left\| {y-x} \right\|_2^{p+1} } \right\}.
\end{equation}
and given a number $L \geq 0$ and point $z \in \R^n$, we define
\begin{equation}
F_{L,z} \left( x \right) \triangleq f\left( x \right)+\frac{L}{2}\left\| 
{x-z} \right\|_2^2 .
\end{equation}

\begin{theorem}
\label{Th:main-conv}
Let sequence $(x^k,y^k,u^k)$, $k \geq 0$ be generated by Algorithm \ref{alg:MSN}. Then 
\[
f(y^N) - f^* \leq \frac{cM_p\|y^0-x_*\|_2^{p+1}}{N^{\frac{3p+1}{2}}}, \text{ } c = \frac{2^{\frac{3(p+1)^2+4}{4}}(p+1)}{p!}.\]

\end{theorem}
Note that this bound allows to obtain an $O\left(\left(\frac{M_pR^{p+1}}{\e}\right)^{\frac{2}{3p+1}} \right)$ iteration complexity. The implementability and cost of each iteration is discussed below in Section \ref{S:Implem_details}.
The proof of Theorem~\ref{Th:main-conv} is based on the framework of \cite{monteiro2013accelerated}, which is presented in the next subsection. 

\floatname{algorithm}{Algorithm}
\begin{algorithm}[H]
\caption{Optimal Tensor Method}
\label{alg:MSN}
\begin{algorithmic}[1]
\REQUIRE $u_0, y_0$~--- starting points; $N$~--- iteration number; $A_0 = 0$
\ENSURE $y^N$
\FOR{$k=0,1,2,\ldots, N-1$}
\STATE Choose $L_k$ such that 
\begin{equation}
\label{eq:L_k-cond}
     \frac{1}{2}\leq \frac{2(p+1)M_p}{p!L_k}\|y^{k+1} - x^k\|_2^{p-1} \leq 1 ,  
\end{equation}
where
$$
a_{k+1} =\frac{1 \mathord{\left/ {\vphantom {1 {L_k }}} \right. 
\kern-\nulldelimiterspace} {L_k }+\sqrt {1 \mathord{\left/ {\vphantom {1 
{L_k^2 }}} \right. \kern-\nulldelimiterspace} {L_k^2 }+4{A_k } 
\mathord{\left/ {\vphantom {{A_k } {L_k }}} \right. 
\kern-\nulldelimiterspace} {L_k }} }{2},
\quad
A_{k+1} =A_k +a_{k+1} , \;\;\; \{\text{note that } L_ka_k^2 = A_{k+1}\}
$$
$$
x^k = \frac{A_k}{A_{k+1}}y^k + \frac{a_{k+1}}{A_{k+1}}u^k, \quad y^{k+1} = T_{p,pM_p}^{F_{L_k,x^k}}(x^k).\; 
$$
\STATE $u^{k+1} = u^k - a_{k+1}\nabla f(y^{k+1})$
\ENDFOR
\RETURN $y^N$
\end{algorithmic}
\end{algorithm}

\subsection{Accelerated hybrid proximal extragradient method}

\cite{monteiro2013accelerated} introduced 
Algorithm~\ref{alg:MS}
for convex optimization problems. To find $y^{k+1}$ on each iteration, the authors use gradient type method for the case $p=1$ and a trust region Newton-type method for the case $p=2$. 
Their analysis of the algorithm is based on the following Theorem.
\begin{theorem}[{ \cite[Theorem 3.6 ]{monteiro2013accelerated} }]
\label{Th:M-S-conv}
Let sequence $(x^k,y^k,u^k)$, $k \geq 0$ be generated by Algorithm \ref{alg:MS} and define $R:=\left\| {y^0-x_\ast } \right\|_2$. Then, for all $N \geq 0$, 
\begin{equation}
\label{eq11}
\frac{1}{2}\left\| {u^N-x_\ast } \right\|_2^2 +A_N \cdot \left( {f\left( 
{y^N} \right)-f\left( {x_\ast } \right)} 
\right)+\frac{1}{4}\sum\limits_{k=1}^N {A_k L_{k-1} \left\| {y^k-x^{k-1}} 
\right\|_2^2 } \le \frac{R^2}{2},
\end{equation}
\begin{equation}
\label{eq12}
f\left( {y^N} \right)-f\left( {x_\ast } \right)\le \frac{R^2}{2A_N },
\quad
\left\| {u^N-x_\ast } \right\|_2 \le R,
\end{equation}
\begin{equation}
\label{eq13}
\sum\limits_{k=1}^N {A_k L_{k-1} \left\| {y^k-x^{k-1}} \right\|_2^2 } \le 
2R^2.
\end{equation}
\end{theorem}

We also need the following Lemma.
\begin{lemma}[ {\cite[Lemma 3.7 a)]{monteiro2013accelerated}}]\label{pr:prop2}
    Let sequences $\left\{ {A_k ,L_k } \right\}$, $k\geq 0$ be generated by Algorithm \ref{alg:MS}. Then, for all $N\geq 0$,
\begin{equation}
\label{eq14}
A_N \ge \frac{1}{4}\left( {\sum\limits_{k=1}^N {\frac{1}{\sqrt {L_{k-1} } }} 
} \right)^2.
\end{equation}    
\end{lemma}

\floatname{algorithm}{Algorithm}
\begin{algorithm}[H]
\caption{Accelerated hybrid proximal extragradient method}
\label{alg:MS}
\begin{algorithmic}[1]
\REQUIRE $u_0, y_0$~--- starting point; $N$~--- iteration number; $A_0 = 0$
\ENSURE $y^N$
\FOR{$k=0,1,2,\ldots, N-1$}
\STATE Choose $L_k$ and $y^{k+1}$ s.t. $\left\| {\nabla F_{L_k ,x^k} \left( {y^{k+1}} \right)} \right\|_2 \le 
\frac{L_k }{2}\left\| {y^{k+1}-x^k} \right\|_2$, where
\[
a_{k+1} =\frac{1 \mathord{\left/ {\vphantom {1 {L_k }}} \right. 
\kern-\nulldelimiterspace} {L_k }+\sqrt {1 \mathord{\left/ {\vphantom {1 
{L_k^2 }}} \right. \kern-\nulldelimiterspace} {L_k^2 }+4{A_k } 
\mathord{\left/ {\vphantom {{A_k } {L_k }}} \right. 
\kern-\nulldelimiterspace} {L_k }} }{2},
\quad
A_{k+1} =A_k +a_{k+1} , \quad x^k=\frac{A_k }{A_{k+1} }y^k+\frac{a_{k+1} }{A_{k+1} }u^k.
\]
\STATE $
u^{k+1}=u^k-a_{k+1} \nabla f\left( {y^{k+1}} \right).$
\ENDFOR
\RETURN $y^N$
\end{algorithmic}
\end{algorithm}

\subsection{Proof of Theorem \ref{Th:main-conv}}
It follows from Algorithm \ref{alg:MSN} that $y^{k+1} = 
T_{p,pM_p}^{F_{L_k,x^k}}(x^k)$, thus by  \cite[Lemma 1]{nesterov2018implementable},
\[
\left\| {\nabla F_{L_k ,x^k} \left( y^{k+1} 
\right)} \right\|_2 \le \frac{\left( {p+1} \right)M_p }{p!}\left\| {y^{k+1}-x^k} \right\|_2^p .
\]
At the same time, by the condition in step 2 of Algorithm, \ref{alg:MSN},
\[
\frac{2(p+1)M_p}{p!L_k}\|y^{k+1} - x^k\|_2^{p-1} \leqslant 1.
\]
Hence,
\[
\left\| {\nabla F_{L_k ,x^k} \left( y^{k+1} 
\right)} \right\|_2 \le \frac{L_k }{2}\left\| {y^{k+1}-x^k} \right\|_2
\]
and we can apply the framework of the previous subsection. What remains is to estimate the growth of $A_N$, which is our next step.

By the condition in step 2 of Algorithm, \ref{alg:MSN}, 
\begin{equation}
\label{eq15}
\frac{1}{L_k }\left\| {y^{k+1}-x^k} \right\|_2^{p-1} \ge \theta ,
\end{equation}
where $\theta = \frac{p!}{4(p+1)M_p}$. Using this inequality, we prove that
\begin{equation}
\label{eq16}
\sum\limits_{k=1}^N {A_k L_{k-1}^{\frac{p+1}{p-1}} } \le 2R^2\theta 
^{-\frac{2}{p-1}}.
\end{equation}
Indeed, from \eqref{eq13} and \eqref{eq15} we have that
\begin{align}
\theta ^{\frac{2}{p-1}}\sum\limits_{k=1}^N {A_k L_{k-1}^{\frac{p+1}{p-1}} } 
    &\le \sum\limits_{k=1}^N {A_k L_{k-1}^{1+\frac{2}{p-1}} \left( 
    {\frac{1}{L_{k-1} }\left\| {y^k-x^{k-1}} \right\|_2^{p-1} } 
    \right)^{\frac{2}{p-1}}} \notag \\
    &=\sum\limits_{k=1}^N {A_k L_{k-1} \left\| {y^k-x^{k-1}} \right\|_2^2 } \le 2R^2.
\end{align}    
Further, from \eqref{eq16} it follows that
\begin{equation}
\label{eq17}
\sum\limits_{k=1}^N {\frac{1}{\sqrt {L_{k-1} } }} \ge \frac{\theta 
^{\frac{1}{p+1}}}{\left( {2R^2} \right)^{\frac{p-1}{2\left( {p+1} 
\right)}}}\left( {\sum\limits_{k=1}^N {A_k^{\frac{p-1}{3p+1}} } } 
\right)^{\frac{3p+1}{2\left( {p+1} \right)}}.
\end{equation}

To prove that, let us introduce new variables $z_k =1 \mathord{\left/ {\vphantom {1 
{\sqrt {L_{k-1} } }}} \right. \kern-\nulldelimiterspace} {\sqrt {L_{k-1} } 
}$ and consider the following optimization problem to find the worst possble value of the l.h.s. in \eqref{eq17}
\begin{equation}
\label{eq18}
\min \sum\limits_{k=1}^N {z_k } \quad \text{s.t.} \quad  \sum\limits_{k=1}^N 
{A_k z_k^{-\gamma } } \le C ,
\end{equation}
where in accordance with \eqref{eq16}
\[
\gamma =2\frac{p+1}{p-1},
\quad
C=2R^2\theta ^{-\frac{2}{p-1}}.
\]
Since the objective and constraints are separable, this problem can be solved explicitly by the Lagrange principle
\[
z_k =\left( {\frac{1}{C}\sum\limits_{j=1}^N {A_j^{\frac{1}{\gamma +1}} } } 
\right)^{1 \mathord{\left/ {\vphantom {1 \gamma }} \right. 
\kern-\nulldelimiterspace} \gamma }A_k^{\frac{1}{\gamma +1}} .
\]
Hence,
\[
\mathop {\min }\limits_{\sum\limits_{k=1}^N {A_k z_k^{-\gamma } } \le C} 
\sum\limits_{k=1}^N {z_k } =\frac{1}{C^{1 \mathord{\left/ {\vphantom {1 
\gamma }} \right. \kern-\nulldelimiterspace} \gamma }}\left( 
{\sum\limits_{k=1}^N {A_k^{\frac{1}{\gamma +1}} } } \right)^{\frac{\gamma 
+1}{\gamma }}.
\]
From this inequality, \eqref{eq14} and \eqref{eq17}, we have
\begin{equation}
\label{eq19}
A_N \ge\frac{1}{4} \frac{\theta 
^{\frac{2}{p+1}}}{\left( {2R^2} \right)^{\frac{p-1}{ {p+1} }}} \left( {\sum\limits_{k=1}^N {A_k^{\frac{p-1}{3p+1}} 
} } \right)^{\frac{3p+1}{ {p+1}}}.
\end{equation}

From this inequality, we obtain that there exists a number $c$ 
such that, for all $N \geq 0$,
\begin{equation}
\label{eq20}
A_N \ge \frac{1}{cM_p R^{p-1}}N^{\frac{3p+1}{2}}.
\end{equation}
The derivation of exact value of the constant $c$ can be found in Lemma~\ref{lem:c_estimate} in Appendix.
This finishes the proof.

\subsection{Implementation details}
\label{S:Implem_details}
First of all, Theorem 1 in \cite{nesterov2018implementable} says that, by the appropriate choice $M=pM_p$ in \eqref{eq4}, the subproblem for finding $y^{k+1}$ in step 2 of Algorithm \ref{alg:MSN} is convex and, thus is tractable. Moreover, for $p=2$ this step corresponds to the step of cubic regularized Newton method of \cite{nesterov2006cubic} and, as it is shown there, can be computed with the same complexity as solving a linear system. For the case $p=3$, \cite{nesterov2018implementable} showed that this step can be also computed efficiently. In both cases the complexity of calculating $y^{k+1}$ is $\tilde {{\rm O}}\left( {n^{2.37}} \right)$.

Let us now discuss the process of finding such $L_k$ that the inequality \eqref{eq:L_k-cond} holds. By construction,
\[
y^{k+1} = \arg \min_{y \in \R^n} \left\{ {\sum\limits_{r=0}^p {\frac{1}{r!}\nabla ^rf\left( x^k \right)} 
\underbrace {\left[ {y-x^k,...,y-x^k} \right]}_r+\frac{pM_p}{\left( {p+1} 
\right)!}\left\| {y-x^k} \right\|_2^{p+1} + \frac{L_k}{2}\|y - x^k\|_2^2} \right\}.
\]
This problem is strongly convex and, thus, has a unique solution for each $L_k >0$. Hence, $y^{k+1}$ is uniquely defined by $L_k$. At the same time, if $L_k \to 0$, $y^{k+1} \to \tilde{y}^k$ with
\[
\tilde{y}^{k} \in \mbox{Arg} \min_{y \in \R^n} \left\{ {\sum\limits_{r=0}^p {\frac{1}{r!}\nabla ^r f\left( x^k \right)} 
\underbrace {\left[ {y-x^k,...,y-x^k} \right]}_r+\frac{pM_p}{\left( {p+1} 
\right)!}\left\| {y-x^k} \right\|_2^{p+1}} \right\}
\]
being a fixed point. Whence,
\[
\frac{2(p+1)M_p}{p!L_k}\|y^{k+1} - x^k\|_2^{p-1} \to + \infty.
\]
On the other hand, if $L_k \to +\infty$, $y^{k+1} \to x^k$ and
\[
\frac{2(p+1)M_p}{p!L_k}\|y^{k+1} - x^k\|_2^{p-1} \to 0.
\]
By the continuity of the dependence of $y^{k+1}$ from $L_k$, we see that there exists such $L_k$ that inequality \eqref{eq:L_k-cond} holds. Appropriate value of $L_k $ can be found by an extended line-search procedure as in \cite[Section 7]{monteiro2013accelerated}. The details of complexity of the line-search can be found in \cite{jiang2018loptimal,bubeck2018near-optimal}, where the authors prove a bound of $\tilde{O}(1)$ calls of $T_{p,pM_p}^{F_{L_k,x^k}}(x^k)$ on each iteration. 

\section{Extension for Uniformly Convex Case}

In this section, we additionally assume that the objective function is uniformly convex of degree $q \geq 2$, i.e., there exists $\sigma_q >0$ s.t.
\begin{equation}
    \label{eq:unif-conv-def}
    f(y) \geq f(x) + \la \nabla f(x) , y-x \ra + \frac{\sigma_q}{q}\|y-x\|_2^q, \forall x,y \in Q.
\end{equation}
We also assume that $q \leq p+1$.
As a corollary,
\begin{equation}
    \label{eq:unif-conv-opt-point}
    f(y) \geq f(x_*) +  \frac{\sigma_q}{q}\|y-x_*\|_2^q, \forall y \in Q,
\end{equation}
where $x_*$ is a solution to problem \eqref{eq:pr_st}. We show, how the restart technique can be used to accelerate Algorithm \ref{alg:MSN} under this additional assumption.

\begin{algorithm}
\caption{Restarted Optimal Tensor Method}
\label{alg:rest-MSN}
\begin{algorithmic}[1]
   \REQUIRE $p$, $M_p$, $q$, $\sigma_q$, $z_0, \Delta_0$ s.t. $f(z^0)-f^* \leq \Delta_0.$
   \FOR{$k=0,1,...$}
			\STATE 
			\begin{equation}
			    \label{eq:rest-MSN-N_k_def}
			    \text{Set} \quad \Delta_k = \Delta_0\cdot2^{-k} \quad \text{and} \quad N_k = \max\left\{\left\lceil \left( \frac{2cM_p q^{\frac{p+1}{q}}}{\sigma_q^{\frac{p+1}{q}}} \Delta_k^{\frac{p+1-q}{q}} \right)^{\frac{2}{3p+1}}\right\rceil,1 \right\}.
			\end{equation}
			\STATE Set $z_{k+1}=y^{N_k}$ as the output of Algorithm \ref{alg:MSN} started from $z_k$ and run for $N_k$ steps.
			\STATE Set $k=k+1$.

			\ENDFOR			
		\ENSURE $z_k$.
\end{algorithmic}
\end{algorithm}

\begin{theorem}
\label{Th:main-unif-conv-conv}
Let sequence $z^k$, $k \geq 0$ be generated by Algorithm \ref{alg:rest-MSN}. Then 
\[
\frac{\sigma_q}{q}\|z_k-x_*\|_2^q \leq f(z_k) - f^* \leq \Delta_0 \cdot 2^{-k},
\]
and the total number of steps of Algorithm \ref{alg:MSN} is bounded by ($c$ is defined in \eqref{eq20})
\[
\left(2c q^{\frac{p+1}{q}}\right)^{\frac{2}{3p+1}} \frac{M_p^{\frac{2}{3p+1}} }{\sigma_q^{\frac{2(p+1)}{q(3p+1)}}} (\Delta_0)^{\frac{2(p+1-q)}{q(3p+1)}} \cdot \sum_{i=0}^k  2^{-i\frac{2(p+1-q)}{q(3p+1)}} + k.
\]

\end{theorem}
\begin{proof}
Let us prove the first statement of the Theorem by induction. For $k=0$ it holds. If it holds for some $k\geq 0$, by the choice of $N_k$, we have that
\[
\frac{cM_p}{N_k^{\frac{3p+1}{2}}} \left( \frac{q \Delta_k}{\sigma_q}\right)^{\frac{p+1}{q}} \leq \frac{\Delta_k}{2}.
\]
By \eqref{eq:unif-conv-opt-point},
\[
\|z_k - x_*\|_2^{p+1} \leq \left( \frac{q (f(z_k)-f^*)}{\sigma_q}\right)^{\frac{p+1}{q}} \leq \left( \frac{q \Delta_k}{\sigma_q}\right)^{\frac{p+1}{q}}
\]
since, by our assumption, $q\leq p+1$.
Combining the above two inequalities and Theorem \ref{Th:main-conv}, we obtain
\[
f(z_{k+1}) -f^* \leq \frac{cM_p\|z_k-x_*\|_2^{p+1}}{N_k^{\frac{3p+1}{2}}} \leq \frac{\Delta_k}{2} = \Delta_{k+1}.
\]
It remains to bound the total number of steps of Algorithm \ref{alg:MSN}. Denote $\tilde{c} = \left(2c q^{\frac{p+1}{q}}\right)^{\frac{2}{3p+1}}$.
\[
\sum_{i=0}^k N_i \leq \tilde{c} \frac{M_p^{\frac{2}{3p+1}} }{\sigma_q^{\frac{2(p+1)}{q(3p+1)}}} \sum_{i=0}^k (\Delta_0\cdot 2^{-i})^{\frac{2(p+1-q)}{q(3p+1)}} + k \leq \tilde{c} \frac{M_p^{\frac{2}{3p+1}} }{\sigma_q^{\frac{2(p+1)}{q(3p+1)}}} (\Delta_0)^{\frac{2(p+1-q)}{q(3p+1)}} \cdot \sum_{i=0}^k  2^{-i\frac{2(p+1-q)}{q(3p+1)}} + k.
\]

\end{proof}
Let us make several remarks on the complexity of the restarted scheme in different settings. It is easy to see from Theorem \ref{Th:main-unif-conv-conv} that, to achieve an accuracy $\e$, i.e. to find a point $\hat{x}$ s.t. $f( \hat{x}) -f^* \leq \e$, the number of tensor steps in Algorithm \ref{alg:rest-MSN} is
\[
O\left(\frac{M_p^{\frac{2}{3p+1}} }{\sigma_q^{\frac{2(p+1)}{q(3p+1)}}} (\Delta_0)^{\frac{2(p+1-q)}{q(3p+1)}} + \log_2 \frac{\Delta_0}{\e}\right), q<p+1, \text{and} \;\;  O\left(\left(\frac{M_p^{\frac{2}{3p+1}} }{\sigma_q^{\frac{2(p+1)}{q(3p+1)}}}+1 \right)\log_2 \frac{\Delta_0}{\e} \right), q=p+1.
\]
Theorem \ref{Th:main-unif-conv-conv} suggests a natural generalization of first- and second-order condition number \cite{nesterov2008accelerating}. If $f$ is such that $q=p+1$, then the complexity of Algorithm \ref{alg:rest-MSN} depends only logarithmically on the starting point and is proportional to
$
\left( \gamma_p \right)^{\frac{2}{3p+1}}
$,
where $\gamma_p = \frac{M_p}{\sigma_{p+1}}$ is the $p$-th order condition number.
Unfortunately, if $q<p+1$, the complexity depends polinomially on the initial objective residual $\Delta_0$, which, in general, is not controlled.

An interesting special case is when $q=2$ and $p\geq 2$, and, as a consequence, $q<p+1$. As it can be seen from Theorem \ref{Th:M-S-conv} (see also \cite{bubeck2018near-optimal}), the sequence, generated by Algorithm~\ref{alg:MSN} is bounded by some $R = O(\|x^0-x_*\|_2)$. Hence, the constant $M_2$ can be estimated as $M_2 \leq M_pR^{p-2}$. At the same time, in \cite[Sect.6]{nesterov2008accelerating}, it is shown that the Cubic regularized Newton method \cite{nesterov2006cubic} has the region of quadratic convergence given by $\{ x: f(x)-f^* \leq \frac{\sigma_2^2}{2M_2^2} \leq \frac{\sigma_2^2}{2M_p^2R^{2(p-2)}}  \}$. To enter this region, Algorithm \ref{alg:rest-MSN} requires 
\begin{equation}
\label{eq:compl_to_quad_conv}
   O\left(\frac{M_p^{\frac{2}{3p+1}} }{\sigma_2^{\frac{p+1}{3p+1}}} (\Delta_0)^{\frac{p-1}{3p+1}} + \log_2 \frac{\Delta_0M_p^2R^{2(p-2)}}{\sigma_2^2} \right) = O\left(\frac{M_p^{\frac{2}{3p+1}} }{\sigma_2^{\frac{p+1}{3p+1}}} (\Delta_0)^{\frac{p-1}{3p+1}} + \log_2 \frac{M_p^2\Delta_0^{p-1}}{\sigma_2^p}\right), 
\end{equation}
where we used inequality $R^2 \leq \frac{2\Delta_0}{\sigma_2}$, which follows from \eqref{eq:unif-conv-opt-point}. After entering the region of quadratic convergence, Algorithm \ref{alg:rest-MSN} can be switched to the Cubic regularized Newton method \cite{nesterov2006cubic}, which has final stage complexity, \cite[Sect. 6]{nesterov2006cubic}
\[
O\left(\log_{3/2}\log_4\frac{\sigma_2^3}{M_2^2\e} \right) = O\left(\log_{3/2}\log_4\frac{\sigma_2^3}{M_p^2R^{2(p-2)}\e} \right).
\]
Summing this inequality and \eqref{eq:compl_to_quad_conv} we obtain the total complexity of this switching procedure to obtain small accuracy $\e$. Note, that the second term in \eqref{eq:compl_to_quad_conv}  is typically dominated by the first one, so we can ignore it without loss of generality.

Finally, let us compare our upper bound with known lower bounds. For the case $p=1$, $q=2$, our complexity bound coincides with lower bound for first-order methods \cite{nemirovsky1983problem,nesterov2004introduction}. \cite{arjevani2018oracle} propose lower bounds  for second-order methods for the case $p=2$, $q=2$ and our complexity bound coincides with their lower bound up to a change of $D = \sqrt{\frac{\Delta_0}{\sigma_2}}$, which is natural as, in this case $f$ is strongly convex. 

\section{Numerical Analysis}\label{sec:experiments}

In this section, we analyze  and compare the performance of Algorithm~\ref{alg:MSN} with the accelerated tensor method proposed in~\cite{nesterov2018implementable}. 

We study the numerical performance for two classes of functions. Initially, an universal parametric family of objective functions, which are difficult for all tensor methods~\cite{nesterov2018implementable} defined as
\begin{align}\label{eq:bad_functions}
        f_m(x) = \eta_{p + 1} \left ( A_m x \right ) - x_1,
\end{align}
where, for integer parameter $p \, \ge \, 1$,
$\eta_{p + 1} (x) = \frac{1}{p+1} \sum\limits_{i = 1}^ n
| x_i |^{p + 1}$, $2 \, \le \, m \, \le \, n$, $x \, \in \, \mathbb{R}^n$,
$A_m$ is the $n \times n$ block diagonal matrix:
\begin{align}
      A_m = 
\left ( \begin{array}{cc}
U_m & 0 \\
0 & I_{n - m}
\end{array}
\right ), \quad \text{with} \quad U_m = 
\left ( \begin{array}{ccccc}
1 & -1 & 0 & \hdots & 0 \\
0 & 1 & -1 &  \hdots & 0 \\
\vdots & \vdots & \ddots &  & \vdots \\
0 & 0 &  \hdots & 1 & -1 \\
0 & 0 &  \hdots & 0 & 1
\end{array}
\right ),  
\end{align}
and $I_n$ is the identity $n \times n$-matrix. For a detailed description of the high-order derivatives of this class of functions, and its optimality properties see~\cite{nesterov2018implementable}.

Figure~\ref{fig:Bad_functions} shows the normalized optimality gap of the iterations generated by the accelerated tensor method from~\cite{nesterov2018implementable} in Figure\ref{fig:Bad_functions}(a), and Algorithm~\ref{alg:MSN} in~Figure\ref{fig:Bad_functions}(b). We denote the minimum function value as $f^*$. For both results we have used $p=3$, and $n=k= \{5,10,15,20,25\}$. These numerical results show that Algorithm~\ref{alg:MSN} requires a much smaller number of iterations than the accelerated tensor method from~\cite{nesterov2018implementable} to reach the same optimality gap, namely $1\cdot10^{-15}$, for the class of ``bad'' functions described in ~\cite{nesterov2018implementable}. For example, for the case where $n=k=25$, Algorithm~\ref{alg:MSN} has reached the desired accuracy in about $100$ iterations, while the accelerated tensor method requires about $1\cdot10^{4}$.

\begin{figure}
\centering
 \subfigure{\includegraphics[width=0.48\textwidth]{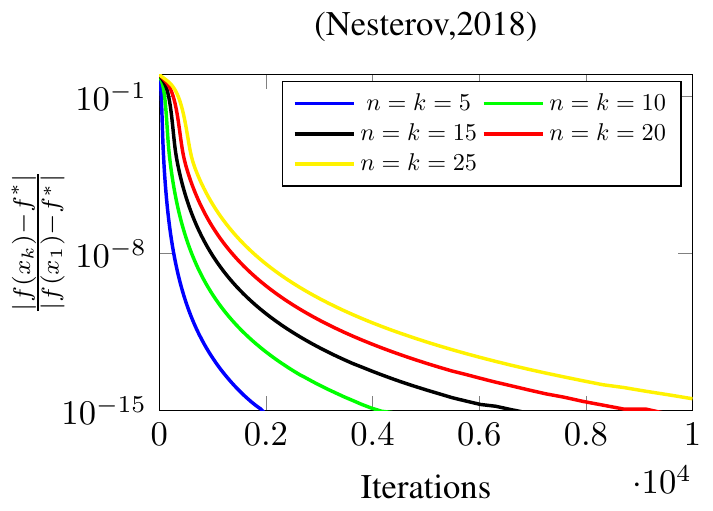}}
 \subfigure{\includegraphics[width=0.48\textwidth]{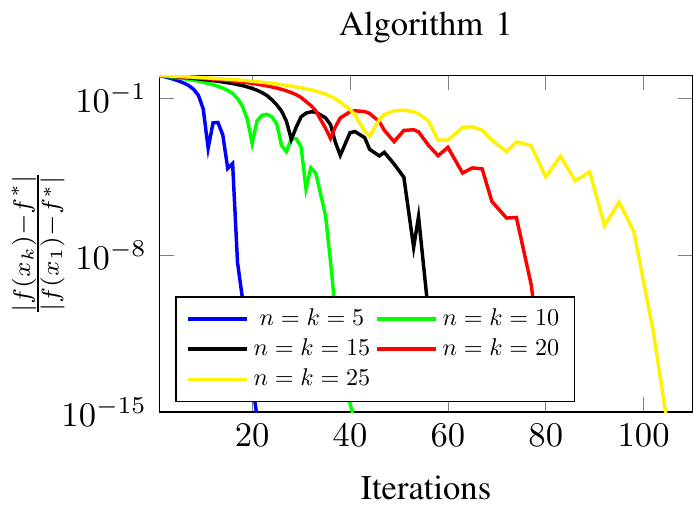}}
 \caption{A performance comparison between the accelerated tensor method in~\cite{nesterov2018implementable} (shown in (a)) and Algorithm~\ref{alg:MSN} (shown in (b)). We minimize an instance of the family of functions in~\eqref{eq:bad_functions} with $p=3$ and various values of dimension $n$ and $k$. Note that the $x$-axis scaling on both figures is different.}
	\label{fig:Bad_functions}
\end{figure}

As a second set of numerical results we study the performance of the proposed method for the non-regularized logistic regression problem. For this problem we are given a set of $d$ data pairs $\{y_i,w_i\}$ for $1\leq i\leq d$, where $y_i \in \{1, -1\}$ is the class label of object $i$, and $w_i \in \mathbb{R}^n$ is the set of features of object $i$. We are interested in finding a vector $x$ that solves the following optimization problem 
\begin{align}\label{eq:logistic loss}
        \frac{1}{d}\sum\limits_{i = 1}^d \ln\Bigl(1 + \exp\bigl(-y_i\langle w_i, x \rangle\bigr)\Bigr) \to \min_{x \in \mathbb{R}^n}. 
\end{align}

Figure~\ref{fig:Logistic} shows the simulation results for the logistic regression problem in~\eqref{eq:logistic loss} for various datasets. Similarly as in Figure~\ref{fig:Bad_functions}, we compare the performance of Algorithm~1, and the accelerated tensor method in~\cite{nesterov2018implementable}. In~Figure~\ref{fig:Logistic}(a) and Figure~\ref{fig:Logistic}(b), we generate synthetic data, where, initially we define a vector $\hat{x} \in [-1,1]$ with every entry is chosen uniformly at random. The set of features for each $i$, i.e., $w_i \in [-1,1]^n$ has also every entry chosen uniformly at random, finally each label is computed as $y_i = \text{sign}(\langle w_i, \hat{x} \rangle)$. For Figure~\ref{fig:Logistic}(a) we set $n=10$ and $d=100$, while in Figure~\ref{fig:Logistic}(b) we set $n=100$ and $d=1000$. Figure~\ref{fig:Logistic}(c) uses the \texttt{mushroom} dataset ($n=8124$ and $d=112$)~\cite{ Dua:2017}, and Figure~\ref{fig:Logistic}(d) uses the \texttt{a9a} dataset ($n=32561$ and $d=123$)~\cite{ Dua:2017}. 

For the logistic regression problem, we don't have access to the optimal value function in general, thus, we plot only the cost function evaluated at the current iterate. As expected by the theoretic results, Algorithm~\ref{alg:MSN} requires one order of magnitude less iterations than the accelerated tensor method from~\cite{nesterov2018implementable} to achieve the same function value.   

In~Appendix~\ref{app:Nesterov_p23}, we numerically compare the performance of the accelerated tensor method from~\cite{nesterov2018implementable} for $p=2$ and $p=3$, as well as its accelerated and non-accelerated versions. 

\begin{figure}
\centering
\includegraphics[width=0.5\textwidth]{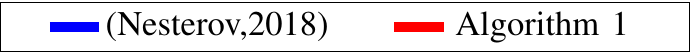}\\
 \subfigure{\includegraphics[width=0.24\textwidth]{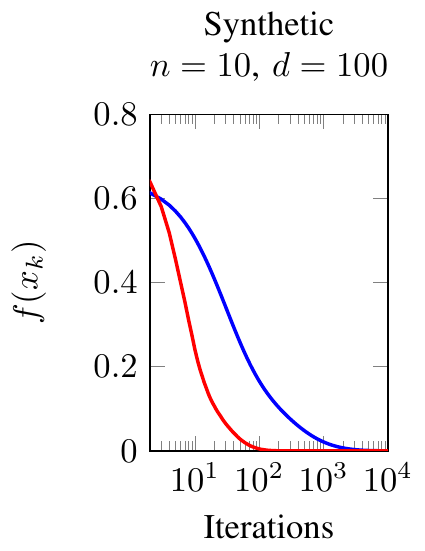}}
 \subfigure{\includegraphics[width=0.24\textwidth]{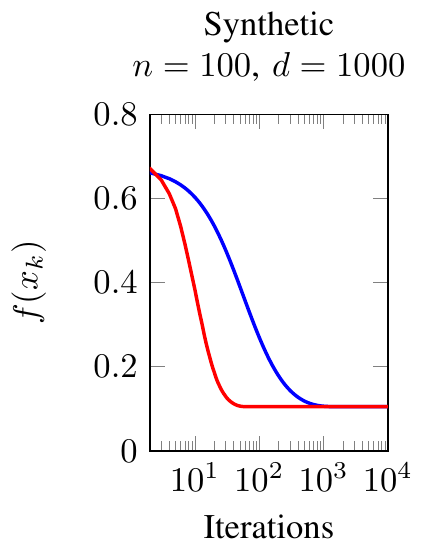}}
  \subfigure{\includegraphics[width=0.24\textwidth]{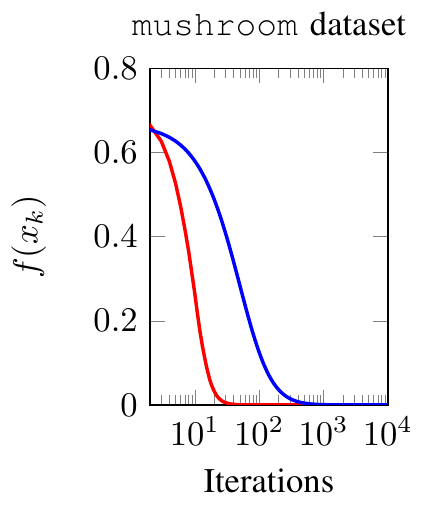}}
 \subfigure{\includegraphics[width=0.22\textwidth]{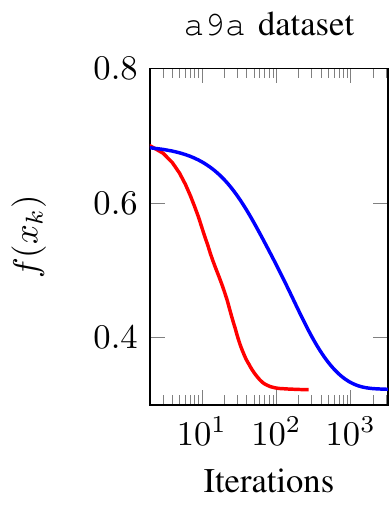}}
 \caption{Performance comparison for the non-regularized logistic regression problem between the accelerated tensor method from~\cite{nesterov2018implementable} and Algorithm~\ref{alg:MSN}. (a) Uses synthetic data with $n=10$ and $d=100$, (b) uses synthetic data with $n=100$ and $d=1000$, (c) uses the \texttt{mushroom} dataset ($d=8124$ and $n=112$)~\cite{ Dua:2017}, and (d) uses the \texttt{a9a} dataset ($d=32561$ and $n=123$)~\cite{ Dua:2017}. }
	\label{fig:Logistic}
\end{figure}

\acks{The authors are grateful to Yurii Nesterov for fruitful discussions.
The work of A.~Gasnikov was supported by RFBR 18-29-03071 mk and was prepared within the framework of the HSE University Basic Research Program and funded by the Russian Academic Excellence Project '5-100', the work of P.~Dvurechensky and E.~Vorontsova was supported by RFBR 18-31-20005 mol-a-ved and the work of E.~Gorbunov was supported by the grant of Russian's President MD-1320.2018.1}

\bibliography{PD_references}
\newpage
\appendix
\part*{Optimal Tensor Methods in Smooth Convex and Uniformly Convex Optimization: Supplementary Material}

\section{Technical lemmas}
\begin{lemma}\label{lem:c_estimate}
    Consider the sequence $\{A_k\}_{k\ge 0}$ of non-negative numbers such that
    \begin{equation}\label{eq:A_N_recurrence_appendix}
        A_N \ge \frac{1}{4}\frac{\theta^{\frac{2}{p+1}}}{\left(2R^2\right)^{\frac{p-1}{p+1}}}\left(\sum\limits_{k=1}^N A_k^{\frac{p-1}{3p+1}}\right)^{\frac{3p+1}{p+1}},
    \end{equation}
    where $p\ge 3$, $\theta = \frac{p!}{4(p+1)M_p}$ and $M_p, R > 0$.
    Then for all $N\ge 0$ we have
    \begin{equation}\label{eq:A_N_lower_bound_appendix}
        A_k \ge \frac{1}{cM_pR^{p-1}}k^{\frac{3p+1}{2}},
    \end{equation}
    where
    \begin{equation}\label{eq:c_estimate_appendix}
        c = \frac{2^{\frac{3(p+1)^2+4}{4}}(p+1)}{p!}
    \end{equation}
\end{lemma}
\begin{proof}
    We prove \eqref{eq:A_N_lower_bound_appendix} by induction. For $k=1$ we have
    \begin{eqnarray*}
        A_1 \overset{\eqref{eq:A_N_recurrence_appendix}}{\ge} \frac{1}{4}\frac{\theta^{\frac{2}{p+1}}}{\left(2R^2\right)^{\frac{p-1}{p+1}}}A_1^{\frac{p-1}{p+1}} \Longleftrightarrow A_1^{\frac{2}{p+1}} \ge \frac{1}{4}\frac{\theta^{\frac{2}{p+1}}}{2^{\frac{p-1}{p+1}}R^{\frac{2(p-1)}{p+1}}}\Longleftrightarrow A_1 \ge \frac{p!}{2^{\frac{3p+5}{2}}(p+1)M_pR^{p-1}}.
    \end{eqnarray*}
    The last inequality implies \eqref{eq:A_N_lower_bound_appendix} for $p\ge 3$.
    Now let us assume that for all $k\le N$ inequality \eqref{eq:A_N_lower_bound_appendix} holds and $N\ge 1$. Next we will establish \eqref{eq:A_N_lower_bound_appendix} for $k = N+1$. We have
    \begin{eqnarray*}
        A_{N+1} &\overset{\eqref{eq:A_N_recurrence_appendix}}{\ge}& \frac{1}{4}\frac{\theta^{\frac{2}{p+1}}}{\left(2R^2\right)^{\frac{p-1}{p+1}}}\left(\sum\limits_{k=1}^{N+1} A_k^{\frac{p-1}{3p+1}}\right)^{\frac{3p+1}{p+1}}\\
        &\ge& \frac{1}{4}\frac{\theta^{\frac{2}{p+1}}}{\left(2R^2\right)^{\frac{p-1}{p+1}}}\left(\sum\limits_{k=1}^{N} A_k^{\frac{p-1}{3p+1}}\right)^{\frac{3p+1}{p+1}}\\
        &\overset{\eqref{eq:A_N_lower_bound_appendix}}{\ge}& \frac{1}{4}\frac{\theta^{\frac{2}{p+1}}}{\left(2R^2\right)^{\frac{p-1}{p+1}}}\left(\left(\frac{1}{cM_pR^{p-1}}\right)^{\frac{p-1}{3p+1}}\sum\limits_{k=1}^{N} k^{\frac{p-1}{2}}\right)^{\frac{3p+1}{p+1}}.
    \end{eqnarray*}
    If $N=1$ then
    \begin{equation}\label{eq:special_case_appendix}
        A_{N+1} = A_2 \ge \frac{1}{2^{\frac{3p+1}{2}}}\frac{\theta^{\frac{2}{p+1}}}{\left(2R^2\right)^{\frac{p-1}{p+1}}}\left(\frac{1}{cM_pR^{p-1}}\right)^{\frac{p-1}{p+1}} (2)^{\frac{3p+1}{2}}.
    \end{equation}
    If $N> 1$ we can write
    \begin{equation}\label{eq:special_case_appendix}
        A_{N+1} \ge\frac{1}{4}\frac{\theta^{\frac{2}{p+1}}}{\left(2R^2\right)^{\frac{p-1}{p+1}}}\left(\frac{1}{cM_pR^{p-1}}\right)^{\frac{p-1}{p+1}}\left(1+\sum\limits_{k=2}^{N} k^{\frac{p-1}{2}}\right)^{\frac{3p+1}{p+1}}.
    \end{equation}
    Since $\frac{p-1}{2} \ge 1$ the function $f(x) = x$ is convex and, as a consequence, we get
    \begin{equation}\label{eq:sum_lower_bound_appendix}
        \sum\limits_{k=2}^N k^{\frac{p-1}{2}} \ge \int\limits_{1}^N x^{\frac{p-1}{2}}dx = \frac{2}{p+1}N^{\frac{p+1}{2}} - \frac{2}{p+1} \ge \frac{2}{p+1}N^{\frac{p+1}{2}} - \frac{1}{2}.
    \end{equation}
    Using this fact we continue:
    \begin{eqnarray*}
        A_{N+1} &\overset{\eqref{eq:sum_lower_bound_appendix}}{\ge}& \frac{1}{4}\frac{\theta^{\frac{2}{p+1}}}{\left(2R^2\right)^{\frac{p-1}{p+1}}}\left(\frac{1}{cM_pR^{p-1}}\right)^{\frac{p-1}{p+1}}\left(\frac{1}{2}+N^{\frac{p+1}{2}}\right)^{\frac{3p+1}{p+1}}\\
        &\ge& \frac{1}{4}\frac{\theta^{\frac{2}{p+1}}}{\left(2R^2\right)^{\frac{p-1}{p+1}}}\left(\frac{1}{cM_pR^{p-1}}\right)^{\frac{p-1}{p+1}} N^{\frac{3p+1}{2}}.
    \end{eqnarray*}
    For all $N > 1$ we have
    \begin{equation*}
        \left(\frac{N}{N+1}\right)^{\frac{3p+1}{2}} = \left(1-\frac{1}{N+1}\right)^{\frac{3p+1}{2}} \ge \left(1-\frac{1}{2}\right)^{\frac{3p+1}{2}} = \frac{1}{2^{\frac{3p+1}{2}}}.
    \end{equation*}
    From this and \eqref{eq:special_case_appendix} we obtain that for all $N\ge1$
    \begin{equation*}
        A_{N+1} \ge \frac{1}{2^{\frac{3p+1}{2}}}\frac{\theta^{\frac{2}{p+1}}}{\left(2R^2\right)^{\frac{p-1}{p+1}}}\left(\frac{1}{cM_pR^{p-1}}\right)^{\frac{p-1}{p+1}} (N+1)^{\frac{3p+1}{2}}.
    \end{equation*}
    It remains to show that \eqref{eq:c_estimate_appendix} implies 
    \begin{equation*}
        \frac{1}{2^{\frac{3p+1}{2}}}\frac{\theta^{\frac{2}{p+1}}}{\left(2R^2\right)^{\frac{p-1}{p+1}}}\left(\frac{1}{cM_pR^{p-1}}\right)^{\frac{p-1}{p+1}} = \frac{1}{cM_pR^{p-1}}.
    \end{equation*}
    Using $\theta = \frac{p!}{4(p+1)M_p}$ we get
    \begin{eqnarray*}
        \frac{1}{2^{\frac{3p+1}{2}}}\frac{\theta^{\frac{2}{p+1}}}{\left(2R^2\right)^{\frac{p-1}{p+1}}}\left(\frac{1}{cM_pR^{p-1}}\right)^{\frac{p-1}{p+1}} = \frac{1}{cM_pR^{p-1}} \Longleftrightarrow c^{\frac{2}{p+1}}\frac{1}{2^{\frac{3p+1}{2}}}\left(\frac{p!}{4(p+1)}\right)^{\frac{2}{p+1}}\frac{1}{2^{\frac{p-1}{p+1}}} = 1\\
        \Longleftrightarrow c^{\frac{2}{p+1}} = 2^{\frac{3p+1}{2}}\left(\frac{4(p+1)}{p!}\right)^{\frac{2}{p+1}}2^{\frac{p-1}{p+1}} \Longleftrightarrow c = 2^{\frac{(3p+1)(p+1)}{4}}\frac{4(p+1)}{p!}2^{\frac{p-1}{2}}\\
        \Longleftrightarrow c = \frac{2^{\frac{3(p+1)^2+4}{4}}(p+1)}{p!},
    \end{eqnarray*}
    which is exactly what we have in \eqref{eq:c_estimate_appendix}.
\end{proof}

\section{Comparison of the accelerated tensor method from~\cite{nesterov2018implementable} for $p=2$ and $p=3$. }\label{app:Nesterov_p23}

In this appendix, we numerically compare the performance of the accelerated tensor method proposed in~\citep{nesterov2018implementable}, for $p=2$ and $p=3$. We also compare the accelerated and non-accelerated version of this method.

Similarly as in Figure~\ref{fig:Bad_functions} and Figure~\ref{fig:Logistic}, we present the numerical results for the class of bad functions defined in~\eqref{eq:bad_functions} and one instance of the logistic regression problem.


\begin{figure}[htb]
\center{\includegraphics[width=0.6\linewidth]{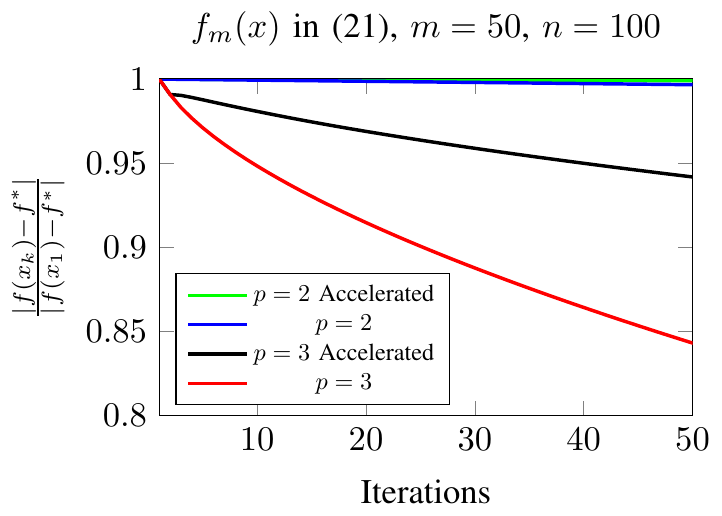}} 
			\caption{
Performance of tensor methods
and accelerated tensor methods for $p = 2$
and $p = 3$
on a difficult instance~\eqref{eq:bad_functions}
for all unconstrained minimization tensor methods with $n=100$ and $m=50$.}
			\label{ris:3}
		\end{figure}

In Figure~\ref{ris:3}, we compare the behavior of
the following methods: 1) tensor method ~\cite{nesterov2018implementable} for $p = 3$; 
2) accelerated tensor method~\cite{nesterov2018implementable} for $p = 3$; 3) tensor method ~\cite{nesterov2018implementable} for $p = 2$; 4) accelerated tensor method~\cite{nesterov2018implementable} for $p = 2$. Again, the optimal
function value is denoted by $f^*$. Interestingly, we obtain that the non-accelerated method outperforms the accelerated method for the first $m$ iterations. Since Theorem~$4$ from \cite{nesterov2018implementable} works only for $k\le m$ we don't study the behaviour of the methods for larger number of iterations. Even in this simple setting it is still non-trivial how to implement tensor methods for such bad examples of functions. 

\begin{figure}[htb]
\center{\includegraphics[width=0.51\linewidth]{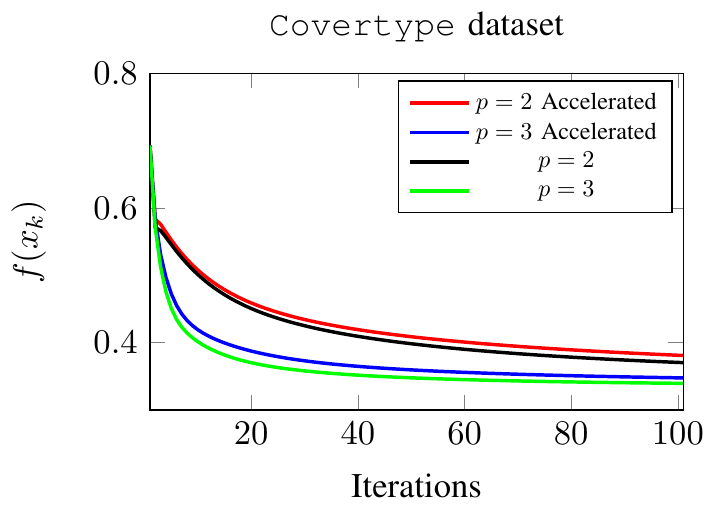} }

			\caption{Function value achieved by the iterates of the accelerated tensor method for the logistic regression problem on the \texttt{Covertype} dataset~\cite{Dua:2017}. Number of samples $d=20000$, dimension $n=55$.}
			\label{ris:4}
\end{figure}

In Figure~\ref{ris:4}, we consider the behaviour of the same set of methods as in Figure~\ref{ris:3}, but for logistic regression problem defined in \eqref{eq:logistic loss} on Covertype dataset \cite{ Dua:2017}.
And again, we notice that in both cases non-accelerated version works better in our experiments 

First of all, we point out that tensor methods in general are non-trivial in implementation, so, it is interesting direction of the future work to get better implementation. Secondly, we conjecture that slow convergence that we see in our experiments is because of large $M_p$ that we use. Due to tuning of the parameters one can obtain better convergence in practice.


\end{document}